\documentclass[11pt]{article}

\usepackage{amsmath,amssymb,amsthm,color,graphicx}
\usepackage{subfigure}

\newtheorem{theorem}{Theorem}
\newtheorem*{theorem*}{Theorem}
\newtheorem{lemma}{Lemma}
\newtheorem{claim}{Claim}

\theoremstyle{definition}

\newcommand{\bbox}{\vrule height7pt width4pt depth1pt}

\newcommand{\comment}[1]{}
\newcommand{\ignore}[1]{}

\def\clap#1{\hbox to 0pt{\hss#1\hss}}

\makeatletter
  \def\moverlay{\mathpalette\mov@rlay}
  \def\mov@rlay#1#2{\leavevmode\vtop{%
    \baselineskip\z@skip \lineskiplimit-\maxdimen
    \ialign{\hfil$#1##$\hfil\cr#2\crcr}}}
\makeatother

\def\F{{\cal F}}
\def\P{{\cal P}}
\def\B{{\cal B}}
\def\D{{\cal D}}
\def\C{{\cal C}}
\def\H{{\cal H}}

\newcommand{\tld}[1]{{\tilde {#1}}}

\newcommand{\remove}[1]{}

\begin{document}
\title{Matchings vs hitting sets among 
half-spaces in low dimensional euclidean spaces}
\author{
Shay Moran\thanks{
Department of Computer Science, Technion---Israel Institute of Technology,
Haifa 32000, Israel, and Max Planck Institute for Informatics, Saarbr\"{u}cken, Germany.
{\tt shaymoran1@gmail.com}.
}
\and
Rom Pinchasi\thanks{
Mathematics Department,
Technion---Israel Institute of Technology,
Haifa 32000, Israel.
{\tt room@math.technion.ac.il}. Supported by ISF grant (grant No. 1357/12).
}
}

\maketitle
\thispagestyle{empty}

\begin{abstract}
Let $\F$ be any collection of linearly separable sets of a set $P$ of $n$ 
points either in $\mathbb{R}^2$, or in $\mathbb{R}^3$.
We show that for every natural number $k$ either one can
find $k$ pairwise disjoint sets in $\F$, or there are $O(k)$ points in $P$ 
that together hit all sets in $\F$. The proof is based on showing a similar 
result for families $\F$ of sets separable by pseudo-discs in $\mathbb{R}^2$.
We complement these statements by showing that analogous result fails to hold 
for collections of linearly separable sets in $\mathbb{R}^4$ and higher 
dimensional euclidean spaces.
\end{abstract}

\newpage
\setcounter{page}{1}

\section{Introduction}

Let $\mathcal{H}=(V,E)$ be a a hyper-graph. A hitting set for $\H$ is a subset of vertices which intersects every edge in $E$. A matching in $\H$ is a subset of mutually disjoint edges. Let $\tau(\H)$ denote the size of a minimum hitting set of $\H$ and let $\nu(\H)$ denote the size of a maximum matching of $\H$. The parameters $\tau(\H),\nu(\H)$ were studied extensively in combinatorics and in computer science. 
$\tau(\H)$ and $\nu(\H)$ relate to each other. Indeed, every hitting set must contain a distinct element from each edge in any matching and therefore $\nu(\H)\leq\tau(\H)$. Moreover, by strong duality for linear programming it follows that the size of a minimum \emph{fractional}\footnote{put a non-negative weight on each vertex so that for every edge, the total weight of all vertices in it is at least $1$ } hitting set, denoted by  $\tau^*(\H)$, is equal to the size of a maximum \emph{fractional}\footnote{put a non-negative weight on each edge so that for every vertex, the total weight of all edges covering it is at most~$1$} matching, denoted by $\nu^*(\H)$. 
So every hyper-graph $\H$ satisfies:
$$\nu(\H)\leq\nu^*(\H)=\tau^*(\H)\leq\tau(H).$$
Hyper-graphs $\H$ for which $\tau(\H)=\nu(\H)$ or for which $\tau(\H)$ and $\nu(\H)$ are close to each other have also been studied. See for example~\cite{AK92,AK95,AKMM02} and references within.

We study the gap between $\nu(\H)$ and $\tau(\H)$ for hyper-graphs $\H$ which can be realized by an arrangement of half-spaces in $\mathbb{R}^{d}$ when $d$ is small. This property is quantified
by the \emph{affine sign-rank}.
The {affine sign-rank} of a hyper-graph $\H$ is the minimum number $d$ for which there is an identification of $V(\H)$ as points in $\mathbb{R}^d$ and of $E(\H)$ as half-spaces in $\mathbb{R}^d$ such that for all $v\in V(\H),e\in E(\H)$, $v\in e$ if and only if the point corresponding to $v$ is in the half-space corresponding to $e$. The affine sign-rank is closely related\footnote{The affine sign-rank is between the sign-rank and the sign-rank plus $1$.} to the sign-rank of $\H$ which was studied in many contexts such as geometry~\cite{AMY14}, machine learning~\cite{FSS01,BES02,FS06}, communication complexity~\cite{PaturiS86,Forster01,FKLMSS01,Sherstov10} and more.

Hyper-graphs with small affine sign-rank have small VC dimension (at most the affine sign-rank plus one) and therefore, by~\cite{BG95,ERS05}, for such hyper-graphs:
$$\tau(\H)\leq O(\tau^*(\H)\log\tau^*(\H)).$$
How about $\nu(\H)$? Is it also close to $\nu^*(\H)$? In general, low VC dimension does not imply that $\nu(\H)$ is close to $\nu^*(\H)$. A simple example is given by $\H=(P,L)$ where $P$ and $L$ are the sets of points and lines in a projective plane of order $n$. 
Recall that in a projective plane of order $n$ $|P|=|L|=n^2+n+1$, each two lines intersect in a unique point, each two points have a unique line containing both of them, each line contain exactly $n+1$ points and each point has exactly $n+1$ lines containing it.
Thus, its VC dimension is $2$, $\nu(\H)=1$ (since every two lines intersect) and $\nu^*(\H)\geq\frac{|L|}{n+1}=\frac{n^2+n+1}{n+1}=\Omega(n)$ as we may choose a $\frac{1}{n+1}$ fraction of every line so that every point is covered exactly once and the total weight of the fractional matching is $\frac{|L|}{n+1}$. 
However, since the affine sign-rank of $\H$ is $\Omega(n^{1/2})$~\cite{FKLMSS01,AMY14} this example does not rule out the possibility that $\tau$ and $\nu$ are close for hyper-graphs of constant affine sign-rank. 

We show that if the affine sign-rank of $\H$ is less than $4$ then $\tau(\H)=\Theta(\nu(\H))$.
We complement this by showing that there are hyper-graphs $\H$ with affine sign-rank $4$ 
such that $\nu(\H)=1$ and $\tau(\H)$ is arbitrarily large.

We note that the fact that $\tau(\H)=\Theta(\nu(\H))$ when the affine sign-rank is $2$ is already known~\cite{CH12}. For completeness we add our alternative proof for it and show how this proof is generalized to capture the case of affine sign-rank $3$.

\section{Our results}

For a set $P$ of points in $\mathbb{R}^d$ and a family $\F$ of ranges
in $\mathbb{R}^d$ we denote by $\H(P,\F)$ the hyper-graph on the set of 
vertices $P$ whose edges consist of the sets $\{P \cap F \mid F \in \F\}$,
without multiplicities. So, the affine sign-rank of $\H$ is $d$
if and only if there is a set $P$ of points in $\mathbb{R}^d$
and a family $\F$ of half-spaces in $\mathbb{R}^d$
such that $\H$ is isomorphic to $\H(P,\F)$.

\subsection{The case of affine sign-rank $2$ and pseudo-discs}
As mentioned above, we show that if $\H$ is a hyper-graph
with affine sign-rank $2$ then $\tau(\H)=\Theta(\nu(\H))$.
In fact, we prove it for a more general class of hyper-graphs:
A family $\C$ of simple closed curves in $\mathbb{R}^2$ is called a family of
pseudo-circles if every two curves in $\C$ are either disjoint or cross
at two points. A family of circles, no two of which touch, is a natural example
for such a family. A family of pseudo-discs is a family of compact sets whose 
boundaries form a family of pseudo-circles.
Natural examples for families of pseudo-discs 
are translates of a fixed convex
set in the plane as well as homothetic copies of a fixed convex set in 
the plane.

Note that if the affine sign-rank of $\H$ is $2$
then there is a set of points $P$ in the plane and a family of pseudo-discs
$\F$ such that $\H$ is isomorphic $\H(P,\F)$ (just replace each half-space by a large enough circular disc).

\begin{theorem}[\cite{CH12}]\label{theorem:p1}
Let $P$ be a set of points in the plane and let $\F$ be a family of 
pseudo-discs. Let $\H$ be the hyper-graph $\H=\H(P,\F)$.
Then for every integer $k \geq 1$ either $\H$ has $k$ pairwise 
disjoint edges, or one can find $O(k)$ points in $P$ that hit
all the edges in $\H$.
\end{theorem}

Theorem~\ref{theorem:p1} implies that every $\H$
with affine sign-rank $2$ has $\tau(\H)=\Theta(\nu(\H))$.
Theorem~\ref{theorem:p1} was proved by Chan and Har-Peled 
in~\cite{CH12}, however the proof that we present here
is based on a different approach. Our methods are useful
also in the case when the affine sign-rank is $3$.
The proof of Theorem \ref{theorem:p1} is based on the following Theorem:

\begin{theorem}\label{theorem:m-discs}
Let $\F$ be a family of pseudo-discs in the plane. Let $P$ be a finite set
of points in the plane and consider the hyper-graph $\H=\H(P,\F)$.
There exists an edge $e$ in $\H$ such that the maximum cardinality of
a matching
among the edges in $\H$ that intersect with $e$ is at most $156$.
\end{theorem}

Theorem~\ref{theorem:m-discs} implies
Theorem~\ref{theorem:p1} as follows.
Apply Theorem~\ref{theorem:m-discs} to find an edge $e$ in $\H$
such that among those edges intersecting it there are at most $156$
pairwise disjoint ones. Delete $e$ and those edges intersecting it from 
$\H$. Repeat this until the graph is empty.
If this continues $k$ steps, then we find $k$ pairwise
disjoint edges. Otherwise, we decompose $\H$ into less than $k$ families,
$\H_{1}, \ldots, \H_{\ell}$,
of edges such that in each family $\H_{i}$ there are at most $156$ pairwise 
disjoint edges. 

We then show that for every $1 \leq i \leq \ell$ the edges in $\H_{i}$
can be pierced by $O(1)$ points. This will conclude the proof of 
Theorem~\ref{theorem:p1}. In order to show that each $\H_{i}$
is indeed pierced by $O(1)$ points, we rely on the techniques of 
Alon and Kleitman in \cite{AK92} by proving a $(p,q)$ Theorem for each of the $H_{i}$
(see the proof of Theorem~\ref{theorem:p1}).

Theorem~\ref{theorem:m-discs} is a discrete version (and therefore also
generalization) of Theorem 1 in \cite{P14}, in which the set $P$ is the entire 
plane. The proof of Theorem~\ref{theorem:m-discs} follows the proof
of Theorem 1 in \cite{P14} with some suitable adjustments.  

The result in Theorem~\ref{theorem:m-discs} (and also Theorem 1 in \cite{P14})
can be interpreted as saying
that in every family of pseudo-discs there is a so called ``small'' 
pseudo-disc.
Indeed, notice that in every family of circular discs, 
the disc of smallest area, $D$,
has the property that the maximum number of mutually disjoint discs
from the family that intersect with it is at most $O(1)$
(see the introduction in \cite{P14} and the references therein 
for more details). Theorem~\ref{theorem:m-discs} implies that the same 
phenomenon happens in every family of pseudo-discs.

The authors of~\cite{CH12}, in
which Theorem~\ref{theorem:p1} was first proved, explicitly 
note that one of the challenges they overcome 
is the absence of a ``smallest pseudo-disc''.
In this paper and in \cite{P14} the existence of such pseudo-disc is proved.
We prove Theorems~\ref{theorem:p1}~and~\ref{theorem:m-discs}
in Section~\ref{sec:proofs1}.

\subsection{The case of affine sign-rank $3$}

\begin{theorem}\label{theorem:p2}
Let $P$ be a set of points in $\mathbb{R}^3$ and let $\F$ be a family of 
half-spaces. Let $\H$ be the hyper-graph $\H=\H(P,\F)$.
Then for every integer $k \geq 1$ either $\H$ has $k$ pairwise 
disjoint edges, or one can find $O(k)$ points in $P$ that hit
all the edges in $\H$.
\end{theorem}

Like in the case of affine sign-rank $2$,
the proof of Theorem~\ref{theorem:p2} is based on the following
theorem that is an analogue of Theorem \ref{theorem:m-discs}:

\begin{theorem}\label{theorem:m-spaces}
Let $P$ be a set of points in $\mathbb{R}^3$ and let $\F$ be a family of 
half-spaces. Let $\H$ be the hyper-graph $\H=\H(P,\F)$.
Then there exists an edge in $\H$ such that the cardinality of the
maximum matching
among the edges in $\H$ intersecting it is at most $156$.
\end{theorem}

We prove Theorems~\ref{theorem:p2} and~\ref{theorem:m-spaces}
in Section~\ref{sec:proofs2}.

\subsection{The case of affine sign-rank $4$}
We show that the analogous result to Theorems~\ref{theorem:p1}~and~\ref{theorem:p2}
fails for affine sign-rank greater than $3$.
\begin{theorem}\label{theorem:gap}
For every $n\in\mathbb{N}$ There exists a set $P$ of $N={n\choose 2}$ points
and a set $\F$ of $n$ half-spaces in $\mathbb{R}^4$ such that:
\begin{enumerate}
\item Every two edges in $\H(P,\F)$ have a non-empty intersection (which implies that $\nu(H)=1$).
\item Any subset of $P$ which pierce all edges in $\H(P,\F)$ has
at least $\frac{n-1}{2}$ points in it (i.e. $\tau(H)\geq\frac{n-1}{2}$).
\end{enumerate}
\end{theorem}
We prove Theorem~\ref{theorem:gap} 
in Section~\ref{sec:proofs3}

\subsection{Connection to $\epsilon$-nets}

Theorems~\ref{theorem:p1}~and~\ref{theorem:p2}
immediately imply a result from \cite{MSW90}
about the existence of an $\epsilon$-net of size linear
in $\frac{1}{\epsilon}$ for hyper-graphs $\H(P,\F)$, where
$\F$ is a family of pseudo-discs in $\mathbb{R}^2$ (hence also the special case where $\F$ is a 
family of half-planes) or half-spaces in $\mathbb{R}^3$.
Indeed, given such a hyper-graph $\H$ and
$\epsilon>0$, 
we delete from $\H$ all the edges of cardinality smaller than $\epsilon |P|$.
Set $k=\frac{1}{\epsilon}$. Notice that now
$\H$ does not contain $k$ pairwise disjoint edges simply because 
every edge is of cardinality greater than $\epsilon |P|$. 
It follows that one can
find $O(k)=O(\frac{1}{\epsilon})$ points in $P$ that meet all the edges in 
$\H$. 

Pach and Tardos~\cite{PachT13} have recently shown that for every $\epsilon>0$
and large enough $n$, there is a collection of $n$ points, $P$, in
$\mathbb{R}^4$ and a collection of half spaces, $\F$, such that every $\epsilon$-net
for $\H(P,\F)$ has size $\Omega(\frac{1}{\epsilon}\log\frac{1}{\epsilon})$.
This corresponds to Theorem~~\ref{theorem:gap}, and in fact implies some variant of it.

\subsection{An algorithmic application}
An immediate algorithmic application of Theorems~\ref{theorem:p1} and~\ref{theorem:p2}
is a polynomial constant factor approximation algorithm for finding maximum matching in  hyper-graphs of the form $\H(\P,\F)$ where $\F$ is a set of pseudo-discs (or half-planes) and $\P\subseteq\mathbb{R}^2$ or $\F$ is a set of half-spaces in $\mathbb{R}^3$ and $\P\subseteq\mathbb{R}^3$. Indeed, given such a hyper-graph $\H$, we can repeatedly find a ``small'' edge $e\in E(\H)$ in the sense of Theorems~\ref{theorem:p1} and~\ref{theorem:p2}, add it to the matching and  then delete $e$ and those edges intersecting it from $\H$ and continue until all the edges of $\H$ are consumed.
The final maximal (with respect to set containment) matching $M$ has size which is at least $\frac{1}{156}$ of the size of a maximum matching. 
We note that Chan and Har-Peled~\cite{CH12} give a PTAS for maximum matching among pseudo-discs, with a different constant, also for the weighted case. 


\section{The case of affine sign-rank $2$ and pseudo-discs}
\label{sec:proofs1}
In this section we prove Theorem~\ref{theorem:p1}
and Theorem~\ref{theorem:m-discs}.

We start with the proof of Theorem~\ref{theorem:m-discs} and then use this
result to prove Theorem~\ref{theorem:p1}. 

An important special case of Theorem \ref{theorem:m-discs} in which the set $P$
is the set of all point in $\mathbb{R}^2$ is shown in \cite{P14}.
The proof of Theorem \ref{theorem:m-discs} will follow the same lines
of the proof in \cite{P14} with some suitable adjustments. 

The idea of the proof is to show that if $B$ is a maximum matching
in $\H$ then on average over all edges $e\in B$ the cardinality of
a maximum matching among the edges in $\H$ that intersects with $e$ is less than $157$.
This means that there exists an edge in $B$ with the desired property.

We will make use of the following lemma that is in fact Corollary 1 in 
\cite{P14}:

\begin{lemma}\label{lemma:sn}
Let $B$ be a family of pairwise disjoint sets in the plane
and let $\F$ be a family of pseudo-discs. Let $D$ be a member of $\F$
and suppose that $D$ intersects exactly $k$ members of $B$ one of which is 
the set $e \in B$.
Then for every $2 \leq \ell \leq k$ 
there exists a set $D' \subset D$ such that $D'$ intersects
$e$ and exactly $\ell-1$ other sets from $B$, and
$\F \cup \{D'\}$ is again a family of pseudo-discs.
\end{lemma}

We will also need the next lemma that is parallel to (and will take the place
of) Lemma 2 in \cite{P14}.

\begin{lemma}\label{lemma:planarity}
Let $\F$ be a family of pseudo-discs in the plane. Let $P$ be a finite set
of points in the plane and consider the hyper-graph $\H=\H(P,\F)$.
Assume $B$ is a subgraph of $\H$ consisting of pairwise disjoint hyper-edges.
Consider the graph $G$ whose vertices correspond to the edges in $B$ and 
connect two vertices $e,e' \in B$
by an edge if there is an edge in $\H$ that has a nonempty intersection
with $e$ and with $e'$ and has an empty intersection with all other edges in 
$B$. Then $G$ is planar.
\end{lemma}  

\noindent {\bf Proof.}
We draw $G$ as a topological graph in the plane as follows.
From every edge $e \in B$ we pick one vertex, that we denote by $v(e)$,
and the collection of all these vertices is the set $V$ of vertices of $G$.
Denote by $\H_{2}$ the set of all edges in $H$ that have a non-empty 
intersection with precisely two
of the edges in $B$.
For every pair of edges $e$ and $e'$ in $B$ that are intersected by
some edge $f$ (possibly such an edge $f$ is not unique) in $\H_{2}$
we draw an edge between $v(e)$ and $v(e')$ as 
follows. 
Pick a vertex  $x \in e \cap f$ and a vertex $x' \in e' \cap f$.
Recall that $f$ is the intersection of 
$P$ with some pseudo-disc $D$ in $\F$. Similarly, let 
$C$ and $C'$ be two pseudo-discs in $\F$ whose intersection with $P$ is 
equal to $e$ and $e'$, respectively.
Let $W_{xx'}$ be an arc, 
connecting $x$ and $x'$, that lies entirely in $D$. 
Let $W_{v(e)x}$ be an arc connecting 
$v(e)$ to $x$ that lies entirely in $C$. Let $W_{v(e')x'}$ be an arc connecting 
$v(e')$ to $x'$ that lies entirely in $C'$.
Finally, we draw the edge in $G$ connecting $v(e)$ and $v(e')$ as the union
(or concatenation) of $W_{v(e)x}, W_{xx'}$, and $W_{x'v(e')}$.
We will show that any two edges in $G$ that do not share a common vertex 
are drawn so that they cross an even number of times. 
The Hanani-Tutte Theorem (\cite{H34,T70})
then implies the planarity of $G$.

We will use the following elementary lemma from \cite{BPR13}:

\begin{lemma}[Lemma 1 in \cite{BPR13}]\label{lemma:BPR13}
Let $D_{1}$ and $D_{2}$ be two pseudo-discs in the plane. Let
$x$ and $y$ be two points in $D_{1} \setminus D_{2}$. 
Let $a$ and $b$ be two points in $D_{2} \setminus D_{1}$. 
Let $\gamma_{xy}$ be any Jordan arc connecting $x$ and
$y$ that is fully contained in $D_{1}$. Let
$\gamma_{ab}$ be any Jordan arc connecting $a$ and $b$ that is fully
contained in $D_{2}$. Then $\gamma_{xy}$ and $\gamma_{ab}$ cross
an even number of times.
\end{lemma}

Let $v(e),v(e')$ and $v(k),v(k')$ be four distinct vertices of $G$.
This means in particular that $e,e',k,$ and $k'$ are four pairwise disjoint
hyper-edges in $B$.
Suppose that $v(e)$ and $v(e')$ are connected by an edge in $G$. 
This means that there are $x \in e$ and $x' \in e'$ and $f \in \H_{2}$ 
such that $x \in e \cap f$ and $x' \in e' \cap f$. 
Let $E,E'$, and $F$ in $\F$ be the pseudo-discs such that 
$e=E \cap P, ~~e'=E' \cap P$, and $f=F \cap P$.
Suppose also that $v(k)$ and $v(k')$ are connected by an edge in $G$. 
This means that there are $y \in k$ and $y' \in k'$ and $q \in \H_{2}$ 
such that $y \in k \cap q$ and $y' \in k' \cap q$.
Let $K,K'$, and $Q$ in $\F$ be the pseudo-discs such that 
$k=K \cap P, ~~k'=K' \cap P$, and $q=Q \cap P$.

By Lemma \ref{lemma:BPR13}, $W_{v(e)x}$ and $W_{v(g)y}$ cross an even number of 
times. Indeed, $E$ contains $v(e)$ and $x$ and does not contain $v(k)$ and $y$.
$K$ contains $v(k)$ and $y$ and does not contain $v(e)$ and $x$.
Similarly, each of $W_{v(e)x}, W_{xx'}$, and $W_{v(e')x'}$ crosses
each of $W_{v(k)y}, W_{yy'}$, and $W_{v(k')y'}$ an even number of times.
We conclude that the edge in $G$ connecting $v(e)$ and $v(e')$
crosses the edge in $G$ connecting $v(k)$ and $v(k')$ an even number of times,
as desired.
\bbox


\bigskip

\noindent {\bf Proof of Theorem \ref{theorem:m-discs}.}
The proof goes almost verbatim as the proof of Theorem 1 in \cite{P14}.
Lemma 2 in \cite{P14} is replaced by the above Lemma \ref{lemma:planarity}.

Let $B$ be a collection of pairwise disjoint edges 
in $\H$ of maximum cardinality and let $n=|B|$. 
For every $e \in B$ denote by $\alpha_{1}(e)$ the size of a maximum matching 
among those edges in $\H$ that intersect with $e$ but with no other edge
in $B$. Denote by $\alpha_{2}(e)$ the size of a maximum matching 
among those edges in $\H$ that intersect with $e$ and with precisely one
more edge in $B$.
Denote by $\alpha_{3}(e)$ the size of a maximum matching 
among those edges in $\H$ that intersect with $e$ and with at least two
more edges in $B$.
Observe that it is enough to show that 
$\sum_{e \in B}\alpha_{1}(e)+\alpha_{2}(e)+\alpha_{3}(e)<157n$.

We first note that for every $e \in B$ we must have $\alpha_{1}(e) \leq 1$.
Indeed, otherwise one can find two disjoint edges $e'$ and $e''$ in $H$ that 
do not intersect with any edge in $B$ but $e$.
The set $B\cup \{e',e''\} \setminus \{e\}$ contradicts that maximality
of $B$. 

Next, we show that $\sum_{e \in B}\alpha_{2}(e) \leq 12n$.
Consider the graph $G$ whose vertices correspond to the edges in $B$ and 
connect two vertices $e,e' \in B$
by an edge if there is an edge in $\H$ that has a nonempty intersection
with $e$ and with $e'$ and has an empty intersection with all other edges in 
$B$. By Lemma \ref{lemma:planarity}, $G$ is planar.
Therefore, $G$ has at most $3n$
edges. For every $e \in B$ denote by $d(e)$ the degree of $e$ in $G$.
Therefore, 
\begin{equation}\label{eq:1}
\sum_{e \in B}d(e) \leq 6n.
\end{equation}

We claim that for every $e$ in $B$ we have $\alpha_{2}(e) \leq 2d(e)$.
Indeed, otherwise by the pigeonhole principle one can find three 
pairwise disjoint edges $g,g'$, and $g''$ in $\H$ and an edge $e'$ in $B$
such that each of $g,g'$, and $g''$ intersects $e$ and $e'$ but no other edge 
in $B$. In this case $B \cup \{g,g',g''\} \setminus \{e,e'\}$ contradicts
the maximality of $B$.

Inequality (\ref{eq:1}) implies now $\sum_{e \in B}\alpha_{2}(e) \leq 12n$.
It remains to show that $\sum_{e \in B}\alpha_{3}(e) < 144n$.
The derivation of this inequality is more involved than the derivation
of the inequalities regarding $\alpha_1,\alpha_2$. We will show that
if it is not the case that $\sum_{e \in B}\alpha_{3}(e) < 144n$, 
then we can derive
an (impossible) embedding of $K_{3,3}$ in the plane.

Denote by $\F_{3}$ the subfamily of $\F$ that consists of pseudo-discs in
$\F$ that intersect with three or more edges in $B$.
Using repeatedly Lemma \ref{lemma:sn} with $\F=\H_{3}$ and with
$\ell=3$, we can find, for every $D \in \H_{3}$ and every $e \in B$
that is intersected by $D$, a (new) pseudo-disc $D^{e} \subset D$
that intersects with $e$ and with exactly two more sets from $B$.
Moreover, the collection of all the new sets $D^{e}$ obtained in this way
is a family of pseudo-discs. We denote this family of pseudo-discs by $\D$.
Let $T$ denote the set of all triples of edges in $B$ that are intersected by 
a pseudo-disc in $\D$.

We denote by $Z$ the collection of all pairs of sets from $B$
that appear together in some triple in $T$.
We claim that $|Z| < 12n$: Pick every set in $B$ with probability
$\frac{1}{2}$. Call a pair $\{e,e'\}$ in $Z$ \emph{good} if 
both $e$ and $e'$ were picked and an edge $f \in B$ such that 
$e,e',$ and $f$ is a triple in $T$ was not picked.
The expected number of good pairs in $Z$ is at least 
$1/8$ of the pairs in $Z$. On the other hand,
by Lemma \ref{lemma:planarity} the set of good pairs in $Z$
is the set of edges of a planar graph 
(on an expected number of $n/2$ vertices) and therefore the 
expected number of good pairs is less than $3\cdot\frac{n}{2}$.

Now consider the graph
$K$ whose set of vertices is the edges in $B$ and whose set of edges
is $Z$.
For every $e \in B$ denote by $d(e)$ the degree of $e$ in this graph.
Notice that, in view of the above, $\sum_{e \in B}d(e)=2|Z| < 24n$.

Fix $e \in B$. Define a graph $K^{e}$ on the set of neighbors
of $e$ in $K$ where we connect two neighbors $e_{1},e_{2}$ of $e$ in $K$ 
by an edge in $K^{e}$
if and only if $\{e,e_{1},e_{2}\}$ is a triple in $T$. This is equivalent
to that there is $D \in \D$ that intersects with $e, e_{1}$, and with $e_{2}$.
Denote by $t(e)$ the number of edges in $K^{e}$. 
By ignoring the set $e$ 
and applying Lemma \ref{lemma:planarity}, we see that $K^{e}$ is planar.
$K^{e}$ has $d(e)$ vertices and is planar and therefore $t(e)<3d(e)$. 

We claim that for every $e \in B$ we must have $\alpha_{3}(e) \leq 2t(e)$.
Indeed, assume to the contrary that  
$\alpha_{3}(e) > 2t(e)$. Then there is a collection $Q$ of at least 
$2t(e)+1$ pairwise disjoint edges of $\H$, each of which has a non-empty
intersection with $e$ and with at least two more edges in $B$.
Because of Lemma \ref{lemma:sn}, every edge in $Q$ must have a non-empty
intersection with $e$ and with at least two edges $e'$ and $e''$ 
that form a pair in $Z$. The hyper-edges $e'$ and $e''$ are therefore 
connected by an edge in $K^{e}$. 
By the pigeonhole principle, because there are only $t(e)$ edges in $K^{e}$
while $|Q| \geq 2t(e)+1$, there exist $e'$ and $e''$ that are 
connected by an edge in $K^{e}$ such that $e, e'$, and $e''$ are all 
intersected by three (pairwise disjoint) edges $g_{1}, g_{2},g_{3} \in \D$.
This is impossible as it gives an embedding of the graph $K_{3,3}$ 
in the plane.
To see this, recall that also the sets $e,e_{1},e_{2}$ are pairwise disjoint.
For every $1 \leq i,j \leq 3$ add a small pseudo-disc surrounding
one point in the intersection of $e_{i}$ and $g_{j}$. Lemma \ref{lemma:planarity}
implies now an (impossible) embedding of $K_{3,3}$ in the plane.

\bigskip

We conclude that 
$$
\sum_{e \in \B}\alpha_{3}(e) < \sum_{e \in \B}2t(e) \leq 
\sum_{e \in \B}6d(e) \leq 6 \cdot 24n=144n.
$$

The proof is now complete as we have 
$$
\sum_{e \in B}\alpha_{1}(e)+\alpha_{2}(e)+\alpha_{3}(e) <
n+12n+144n=157n
$$

This implies the existence of $e \in B$
such that $\alpha_{1}(e)+\alpha_{2}(e)+\alpha_{3}(e) \leq 156$. 
\bbox

\bigskip

Having proved Theorem~\ref{theorem:m-discs}, we are now ready to prove 
Theorem~\ref{theorem:p1}.

\noindent {\bf Proof of Theorem~\ref{theorem:p1}.}
Repeatedly apply Theorem \ref{theorem:m-discs} and find an edge $e$ in $\H$
such that among those edges intersecting it there are at most $156$
pairwise disjoint ones. Then delete $e$ and those edges intersecting it from 
$\H$ and continue. If we can continue $k$ steps, then we find $k$ pairwise
disjoint edges. Otherwise, we decompose $\H$ into less than $k$ families,
$\H_{1}, \ldots, \H_{\ell}$,
of edges such that in each family $\H_{i}$ there are at most $156$ pairwise 
disjoint edges. 

We will now show that for every $1 \leq i \leq \ell$ the edges in $\H_{i}$
can be pierced by $O(1)$ points. This will conclude the proof of 
Theorem~\ref{theorem:p1}.

Our strategy is to show that the edges in $\H_{i}$ have the so called
$(p,q)$ property for some $p$ and $q$. That is, out of every $p$ sets in 
$\H_{i}$ there are $q$ that have a non-empty intersection. In fact, by the 
definition of $\H_{i}$, it has the $(157,2)$ property because there are
at most $156$ sets in $H_{i}$ that are pairwise disjoint. This is the first 
step. The next step is to show a $(p,q)$ (for the same $q$ above,
that is $q=2$) theorem for hyper-graphs
$\H(\P,\F)$ where $\F$ is a family of pseudo-discs. This means that we will need
to show that for a family of pseudo-discs $\F$ if $\H(\P,\F)$
has the $(p,q)$ property, then one can find a constant number 
of points in $\P$ that together pierce all edges in $\H(\P,\F)$.

In order to complete the second step we will rely on the techniques of 
Alon and Kleitman in \cite{AK92}. Rather than repeating their proof
and adjusting it to our case, we observe, following Alon et. al in 
\cite{AKMM02} and Matou\v{s}ek in \cite{M04} that it is enough to show
that the edges of $\H(\P,\F)$ have fractional Helly number $2$ (see below) and
have a finite VC-dimension, which implies the existence of an $\epsilon$-net
of size that depends only on $\epsilon$. These two ingredients are enough
to show that $\H(\P,\F)$ has a $(p,2)$ theorem for every $p>2$. 

\bigskip

We recall that a hyper-graph $\H$ is said to have 
a fractional Helly number $k$ if for every $\alpha>0$ there is $\beta>0$
such that for any $n$ and any collection of $n$ sets in $\F$ in which there are 
at least $\alpha {n \choose k}$ $k$-tuples that have nonempty intersection
one can find a point incident to at least $\beta n $ of the sets. 
Here $\beta$ may depend only on $\alpha$ (and the hyper-graph $\H$) but not on $n$.
In our setting the hyper-graph $\H$  is of the form $\H(\P,\F)$ where $\F$ is a set of pseudo discs and $\P$ is a set of points. We will see that every such $\H$ has fractional Helly number  $2$ and that the corresponding $\beta$ does not depend on $\P$ nor on $\F$ (it will only depend on certain combinatorial properties that are possessed by every family of pseudo-discs).

%

We recall also the notion of \emph{union complexity} of a family of sets.
We denote by $U_{\F}(m)$ the maximum complexity (that is, number of faces of
all dimensions)
of the boundary of the union of any $m$ members of $\F$. We will need the 
following well known result from \cite{KLPS86} saying that 
for a family $\F$ of pseudo-discs we have $U_{\F}(m) \leq 12m$

We will use the following theorem from \cite{P15} (see Theorem 1 there)
relating the notion of fractional Helly number with that of union complexity.

\begin{theorem}\label{theorem:p}
Let $g:\mathbb{R} \rightarrow \mathbb{R}$ be a function such that
$\lim_{x \rightarrow \infty}g(x)=0$. 
Suppose that $\F$ is a family of geometric objects in $\mathbb{R}^d$ in 
general position, (that is, no point belongs to the intersection of more than
$d$ boundaries of sets in $\F$) 
such that $U_{\F}(m) \leq g(m)m^k$ for every 
$m \in \mathbb{N}$.
Then for every set of points $P$ the family 
$\F_{P}$ has fractional Helly number at most $k$ 
and this is in a way that depends only on the function $g$ and not on $\F$
or $P$.

To be more precise, 
for every $\alpha>0$ there is a $\beta>0$
such that for any family $\F$ satisfying the conditions in the theorem
and a set of points $P$ in $\mathbb{R}^d$ the following is true:
For any collection of $n$ sets in 
$\H(\F,\P)$ in which there are 
at least $\alpha {n \choose k}$ $k$-tuples that have nonempty intersection
one can find a point in $P$ incident to at least $\beta n $ of the sets.  
\end{theorem}

Theorem \ref{theorem:p} (with $d=2$ and $k=2$) 
and the linear bound on the union complexity of
pseudo-discs \cite{KLPS86} imply that $\H(P,\F)$ 
has fractional Helly number at most $2$. (Notice that we may assume without
loss of generality that the sets in $\F$ are indeed 
in general position and therefore
Theorem \ref{theorem:p} applies here.)

It is well known and not hard to show (see for example Theorem 9 in
\cite{BPR13}) that for a family $\F$ of pseudo-discs and a set $P$ of points
the hyper-graph $\H(P,\F)$  
has a bounded VC-dimension (in fact at most $3$).
Therefore, each $\H_{i}$ has an $\epsilon$-net of size that depends only 
on $\epsilon$ (see \cite{HW87}).
The method of Alon and 
Kleitman \cite{AK92} implies that each $\H(P,\F)$ satisfies a $(p,2)$ theorem.
That is, if any subset $S$ of edges in $\H(P,\F)$ satisfies the $(p,2)$ 
property 
(from every $p$ sets in $S$ there are $2$ sets that intersect), 
then there are $c(p)$ (a constant that depends only on $p$) 
vertices that together pierce all the sets in 
$S$ (see Theorem 4 and the discussion around it in \cite{M04}). 

By our assumption each, $\H_{i}$ has the $(p,2)$-property for $p=157$. 
It follows that one can find a set of points of cardinality
at most $c(157)k$ that together intersect all the edges in $\H$.
\bbox


\section{The case of half-spaces in $\mathbb{R}^3$.}
\label{sec:proofs2}

In this section we prove Theorem~\ref{theorem:p2}.
The proof follows the same trajectory as the proof of Theorem~\ref{theorem:p1}
with analogous lemmata. Technically, the challenge in this case is to derive the analogous
lemmata for half-spaces in $\mathbb{R}^3$.
 
For the proof of Theorem \ref{theorem:p2} we will need a corresponding
three dimensional version of Lemma~\ref{lemma:planarity}:

\begin{lemma}\label{lemma:planarity3d}
Let $\F$ be a family of half-spaces in $\mathbb{R}^3$. Let $P$ be a finite set
of points in $\mathbb{R}^3$ and consider the hyper-graph 
$\H=\H(P,\F)$.
Assume $B$ is a subgraph of $\H$ consisting of pairwise disjoint hyper-edges.
Consider the graph $G$ whose vertices correspond to the edges in $B$ and 
connect two vertices $e,e' \in B$
by an edge if there is an edge in $\H$ that has a nonempty intersection
with $e$ and with $e'$ and has an empty intersection with all other edges in 
$B$. Then $G$ is planar.
\end{lemma} 

\noindent {\bf Proof.}
We notice that if the points of $P$ are in (strictly) convex position,
then Lemma \ref{lemma:planarity3d} follows almost right away from Lemma
\ref{lemma:planarity}. To see this let $S$ denote the convex hull of $P$
and for every half-space $F$ in $\F$ let $F^S$ denote the intersection of
$F$ with the boundary of $S$. Then the collection $\{F^S \mid F \in \F\}$
is a family of pseudo-discs lying on the boundary of $S$. 
Now Lemma \ref{lemma:planarity3d} follows from 
Lemma \ref{lemma:planarity} that, although stated in the plane, applies also to
the boundary of $S$ (homeomorphic to the two dimensional sphere).  

When the points of $P$ are not in convex position such a simple reduction 
is not possible anymore. Nevertheless, we will be able to
make use of Lemma \ref{lemma:planarity} after some suitable modifications.

Denote by $M$ the union of all edges in $B$.
We say that a point of $M$ is \emph{extreme} if it lies on the boundary of the
convex hull of $M$.

\begin{lemma}\label{lemma:h}
Let $e_{1}$ and $e_{2}$ be two edges in $B$.
Suppose that there exists an edge $f \in \H$ such that $f$ has a nonempty
intersection with $e_{1}$ and with $e_{2}$ and $f$ does not intersect any
other edge in $B$. Then there exists a half-space $F'$, not necessarily in 
$\F$, such that both intersections of $\F'$ with $e_{1}$ and with $e_{2}$  
contain extreme points of $M$ and still $F'$ does not intersect any other edge 
in $B$ but $e_{1}$ and $e_{2}$.
\end{lemma}  

\noindent {\bf Proof.}
We shall use the following basic fact several times: Any half-space that has a non-empty
intersection with $M$ contains an extreme point of it.
Let $F$ denote the half-space in $\F$ such that $f=F \cap P \supset F \cap M$.
$F$ contains at least one extreme vertex of $M$.
Because $F \cap M \subset e_{1} \cup e_{2}$ we conclude that there is
an extreme vertex of $M$ either in $F \cap e_{1}$, or in $F \cap e_{2}$
(if there is an extreme vertex of $M$ in both, then we are done with $F'=F$).
Without loss of generality assume that $F \cap e_{2}$ contains an
extreme vertex of $M$. Let $E_{1} \in \F$ be the half-space such that 
$e_{1}=E_{1} \cap P$. $E_{1}$ contains an extreme vertex of $M$ that belongs to
$e_{1}$. Let $\ell$ denote the line of intersection of the boundaries of $F$ 
and $E_{1}$. Notice that $(F \cup E_{1}) \cap M \subset e_{1} \cup e_{2}$.
Take $F'=F$ and start rotating $F'$ about the line $\ell$ such that 
at each moment $F' \subset F \cup E_{1}$. 
At each moment of the rotation until $F'$ coincides with $E_{1}$, the half-space
$F'$ contains the intersection $F \cap E_{1}$ and therefore
$F'$ has a nonempty intersection with $e_{1}$. 
We stop at the last moment where $F'$ still contains an extreme vertex of $M$
that belongs to $e_{2}$. At this moment $F'$ must also contain a vertex of 
$e_{1}$ that is extreme in $M$. This is because at each moment $F'$ 
must contain an extreme vertex of $M$. This completes the proof of the lemma.
\bbox

\bigskip

Going back to the proof of Lemma \ref{lemma:planarity3d},
let $S$ denote the convex hull of $M$. For very edge $e$ in $B$ let 
$F(e) \in \F$ be the half-space in $\F$ such that $e=F(e) \cap P$.
Denote by $\tld{e}$ the set of extreme vertices of $M$ in $e$.
Notice that for every $e \in B$ we have $\tld{e} \neq \emptyset$
because every edge in $B$ is the intersection of $P$ with some 
half-space (in $\F$).
Let $\tld{M}$ denote the set of extreme points in $M$. 
Because $M$ is just the union of all edges in $B$, we have 
$\tld{M}=\cup_{e \in B}\tld{e}$.
Observe that $\{\tld{e} \mid e \in B\}$ is the set of edges
of the hyper-graph $\tld{\H}=\H(\tld{M},\{F(e) \mid e \in B\})$.
For every pair of hyper edges $e,e' \in B$ that are neighbors in the graph
$G$ (defined in the statement of Lemma \ref{lemma:planarity3d}) let 
$F(e,e') \in \F$ denote some half-space in $\F$ that 
has a nonempty intersection only with the edges $e$ and $e'$ from $B$.
By Lemma \ref{lemma:h}, there exists a half-space that, with a slight abuse 
of notation, we denote by $F(\tld{e},\tld{e'})$, not necessarily
in $\F$, such that $F(\tld{e},\tld{e'})$ has a non-empty intersection 
only with $\tld{e}$ and with $\tld{e'}$ from the collection 
$\{\tld{f} \mid f \in B\}$. 

Let 
$$
\F'=\{F_{e} \mid e\in B\} \cup \{F(\tld{e},\tld{e'}) \mid 
\mbox{$(e,e')$ is an edge in $G$}\}.
$$

We define now a graph $G'$ whose set of vertices is 
$B'=\{\tld{e} \mid e \in B\}$.
We connect $\tld{e}$ and $\tld{e'}$ in $B'$ by an edge in $G'$ if there
is an edge $f$ in the hyper-graph $\H(\tld{M}, \F')$ such that 
$f$ has a nonempty intersection with $\tld{e}$ and with $\tld{e'}$ and
$f$ has an empty intersection with all other sets in $B'$.
It follows from the discussion above that if $e$ and $e'$ are two sets 
in $B$ that are connected by an edge in $G$, then 
$\tld{e}$ and $\tld{e'}$ in $B'$ are connected by an edge in $G'$.

Because $\tld{M}$ is in convex position, 
the hyper-graph $\H(\tld{M},\F')$ can be presented
as a hyper-graph on the set of vertices $\tld{M}$ whose set
of edges correspond to pseudo-discs on $S$, where $S$ is 
the boundary of the convex hull of $M$.
We then apply Lemma \ref{lemma:planarity} (where $B$ is replaced by
$\{\tld{e} \mid e \in B\}$ and $\F$ is replaced by $\F'$) and conclude that
$G'$ is planar. The planarity of $G$ follows because $G$ is a subgraph of $G'$.
\bbox

\bigskip

We are now ready to prove Theorem~\ref{theorem:m-spaces}. The proof will
follow the lines and will have a similar structure as of the proof of the 
corresponding theorem for pseudo-discs in the plane, 
namely Theorem \ref{theorem:m-discs}.

\bigskip

\noindent {\bf Proof of Theorem~\ref{theorem:m-spaces}.}
As in the proof of Theorem \ref{theorem:m-discs}, let $B$ be a maximum
(in cardinality) collection of pairwise disjoint edges in $\H$ and let $n=|B|$. 
For every $e \in B$ denote by $\alpha_{1}(e)$ the maximum cardinality of a 
matching among those edges in $\H$ that intersect with $e$ but with no 
other edge in $B$. 
Denote by $\alpha_{2}(e)$ the maximum cardinality of a matching
among those edges in $\H$ that intersect with $e$ and with precisely one
more edge in $B$.
Denote by $\alpha_{3}(e)$ the maximum cardinality of a matching 
among those edges in $\H$ that intersect with $e$ and with at least two
more edges in in $B$.
It is enough to show that 
$\sum_{e \in B}\alpha_{1}(e)+\alpha_{2}(e)+\alpha_{3}(e)<157n$.

For every $e \in B$ we must have $\alpha_{1}(e) \leq 1$, or else we get
a contradiction to the maximality of $B$
(as in the proof of Theorem \ref{theorem:m-discs}).

Next we show that $\sum_{e \in B}\alpha_{2}(e) \leq 12n$.
Consider the graph $G$ whose vertices correspond to the edges in $B$ and 
connect two vertices $e,e' \in B$
by an edge if there is an edge in $\H$ that has a nonempty intersection
with $e$ and with $e'$ and has an empty intersection with all other edges in 
$B$. By Lemma \ref{lemma:planarity3d}, $G$ is planar.
Therefore, $G$ has at most $3n$
edges. For every $e \in B$ denote by $d(e)$ the degree of $e$ in $G$.
Therefore, 
\begin{equation}\label{eq:13d}
\sum_{e \in B}d(e) \leq 6n.
\end{equation}

We claim that for every $e$ in $B$ we have $\alpha_{2}(e) \leq 2d(e)$.
Indeed, otherwise, by the pigeonhole principle, one can find three 
pairwise disjoint edges $g,g'$, and $g''$ in $\H$ and an edge $e'$ in $B$
such that each of $g,g'$, and $g''$ intersects $e$ and $e'$ but no other edge 
in $B$. In this case $B \cup \{g,g',g''\} \setminus \{e,e'\}$ contradicts
the maximality of $B$. 
Inequality (\ref{eq:13d}) implies now $\sum_{e \in B}\alpha_{2}(e) \leq 12n$.

It remains to show that $\sum_{e \in B}\alpha_{3}(e)<144n$.
Denote by $\F_{3}$ the subfamily of $\F$ that consists of half-spaces in
$\F$ that intersect with three or more edges in $B$.
Like in the proof of Theorem~\ref{theorem:m-discs} this part is more involved.
Similarly, we will show that if it is not the case that $\sum_{e \in B}\alpha_{3}(e)<144n$,
then we derive an (impossible) embedding of $K_{3,3}$ in an arrangement of hyper-planes in $\mathbb{R}^3$ (see Claim~\ref{claim:aux}).

For every $F \in \F_{3}$ and every $e \in B$
that is intersected by $F$, we find a (new) half-space $F^{e}$
that intersects with $e$ and with exactly two more edges in $B$.
To do this, let $v \in e$ be an extreme vertex of $P$ and let $h$
be a hyper-plane supporting the convex hull of $P$ at $v$.
Let $\ell$ denote the line of intersection of $h$ and the boundary of $F$.
Rotate $F$ about the line $\ell$ until $F$ intersects only three edges in $B$
one of which must be $e$ because at all times of rotation we have $v \in F$.

We denote the family of all new half-spaces obtained this way by $\D$.
Let $T$ denote the set of all triples of edges in $B$ that are intersected by 
half-spaces in $\D$.

We denote by $Z$ the collection of all pairs of sets from $B$
that appear together in some triple in $T$.
One can show that $|Z| < 12n$: Pick every set in $B$ with probability
$\frac{1}{2}$. Call a pair $\{e,e'\}$ in $Z$ \emph{good} if 
both $e$ and $e'$ were picked and an edge $f \in B$ such that 
$e,e',$ and $f$ is a triple in $T$ was not picked.
The expected number of good pairs in $Z$ is at least 
$1/8$ of the pairs in $Z$. On the other hand, 
by Lemma \ref{lemma:planarity3d} the set of good pairs in $Z$
is the set of edges of a planar graph 
(on an expected number of $n/2$ vertices).
(We refer the reader to the proof of Theorem \ref{theorem:m-discs} to see
this argument a bit more detailed.)

Now consider the graph
$K$ whose set of vertices is the edges in $B$ and whose edges
are those pairs in $Z$.
For every $e \in B$ denote by $d(e)$ the degree of $e$ in this graph.
Notice that, in view of the above, $\sum_{e \in B}d(e)=2|Z| < 24n$.

Fix $e \in B$. Define a graph $K^{e}$ on the set of neighbors
of $e$ in $K$ where we connect two neighbors $e_{1},e_{2}$ of $e$ in $K$ 
by an edge in $K^{e}$
if and only if $\{e,e_{1},e_{2}\}$ is a triple in $T$. This is equivalent
to that there is $D \in \D$ that intersects with $e, e_{1}$, and with $e_{2}$.
Denote by $t(e)$ the number of edges in $K^{e}$. 
By ignoring the set $e$ 
and applying Lemma \ref{lemma:planarity3d}, we see that $K^{e}$ is planar.
$K^{e}$ has $d(e)$ vertices and is planar and therefore $t(e)<3d(e)$. 

We claim that for every $e \in B$ we must have $\alpha_{3}(e) \leq 2t(e)$.

Indeed, assume to the contrary that  
$\alpha_{3}(e) > 2t(e)$. Then there is a collection $Q$ of at least 
$2t(e)+1$ pairwise disjoint edges of $H$, each of which has a non-empty
intersection with $e$ and with at least two more edges in $B$.
Every edge in $Q$ has a non-empty
intersection with $e$ and with at least two edges $e'$ and $e''$ 
that form a pair in $Z$. The hyper-edges $e'$ and $e''$ are therefore 
connected by an edge in $K^{e}$. 
By the pigeonhole principle, because there are only $t(e)$ edges in $K^{e}$
while $|Q| \geq 2t(e)+1$, there exist $e'$ and $e''$ that are 
connected by an edge in $K^{e}$ such that $e, e'$, and $e''$ are all 
intersected by three (pairwise disjoint) edges $g_{1}, g_{2},g_{3} \in Q 
\subset \H$. 
We claim that this situation is impossible. This follows directly from the 
following claim

\begin{claim}\label{claim:aux}
It is impossible to find three half-spaces $u_{1},u_{2},u_{3}$ in 
$\mathbb{R}^3$ and another three half-spaces $w_{1},w_{2},w_{3}$
such that there are nine points $q_{ij}$ for $1\leq i,j \leq 3$
satisfying $q_{ij}$ lies only in $u_{i}$ and $w_{j}$ from the half-spaces
$u_{1},u_{2},u_{3},w_{1},w_{2},w_{3}$.
\end{claim}

\noindent {\bf Proof.}
Considering the dual problem, it is enough to show that one cannot find
three points $u_{1},u_{2},u_{3}$ in 
$\mathbb{R}^3$ and another three points $w_{1},w_{2},w_{3} \in \mathbb{R}^3$
such that there for every $1\leq i,j \leq 3$
there is a half-space containing only $u_{i}$ and $w_{j}$
from the points $u_{1},u_{2},u_{3},w_{1},w_{2},w_{3}$. 

Without loss of generality we assume that all the points are in general position.
We may also assume that one of the triangles $\Delta u_1u_2u_3$
or $\Delta w_1w_2w_3$ is not a face of the convex hull of $\{u_1,u_2,u_3,w_1,w_2,w_3\}$.
Otherwise, the points $u_{1},u_{2},u_{3},w_{1},w_{2},w_{3}$ are in convex
position and each of the segments $[u_i,w_j]$ is an edge of this convex polytope (because by assumption each pair of vertices $w_i,u_j$ is separable from the rest of the vertices by a hyper-plane). The skeleton graph of a three dimensional convex polytope is planar and therefore cannot contain $K_{3,3}$ as a subgraph.
Therefore, without loss of generality we assume that that
the hyper-plane through  $u_{1},u_{2},$ and $u_{3}$ 
separates two of the points $w_{1},w_{2},$ and $w_{3}$.
Let $h$ denote this hyper-plane and assume
without loss of generality that $w_{1}$ and $w_{2}$ lie above $h$ while
$w_{3}$ lies below $h$. We observe that the line through $w_{1}$ and $w_{2}$
must cross triangle $\Delta u_{1}u_{2}u_{3}$ for otherwise
$u_{1},u_{2},u_{3},w_{1},w_{2}$ are in convex position and the edge-graph
of their convex hull is the non-planar $K_{5}$.
Without loss of generality assume that $w_{1}$ lies closer than $w_{2}$
to triangle $\Delta u_{1}u_{2}u_{3}$. 
Denote by $O$ the point of intersection of the line through $w_{1}$ and $w_{2}$
with $h$. For $i=1,2,3$ let $Q_{i}$ be a half-space
containing only $w_{1}$ and $u_{i}$ from 
$u_{1},u_{2},u_{3},w_{1},w_{2},w_{3}$. Observe that all three half-spaces
$Q_{1},Q_{2},$ and $Q_{3}$ must contain the point $O$ (as they separate 
$w_{1}$ and $w_{2}$) and, assuming $h$ is horizontal, their supporting hyper-planes
must all lie above $O$.
This implies that $Q_{1},Q_{2},$ and $Q_{3}$ cover the whole half-space below
$h$ which is impossible as none of $Q_{1},Q_{2},$ and $Q_{3}$ may contain $w_{3}$.  
%
%
\bbox 

\noindent {\bf Remark.} Although it is tempting to believe
that the collection 
of all $2$-sets (that is, sets of two points separable by a half-space)
of a set of points in $\mathbb{R}^3$ is the set of edges
of a planar graph, this is not the case. One can check that $K_{5}$ can be 
realized in this way. Claim \ref{claim:aux} shows that $K_{3,3}$
cannot be realized in this way. 

\bigskip

Going back to the proof of Theorem \ref{theorem:m-spaces}, we have:
$$
\sum_{e \in \B}\alpha_{3}(e) < \sum_{e \in \B}2t(e) \leq 
\sum_{e \in \B}6d(e) \leq 6 \cdot 24n=144n.
$$

The proof is now complete as we have 
$$
\sum_{e \in B}\alpha_{1}(e)+\alpha_{2}(e)+\alpha_{3}(e) <
n+12n+144n=157n,
$$

and this implies the existence of $e \in B$
such that $\alpha_{1}(e)+\alpha_{2}(e)+\alpha_{3}(e) \leq 156$. 
\bbox

\bigskip

In the same way that Theorem \ref{theorem:p1} is a corollary of
Theorem \ref{theorem:m-discs}, we conclude Theorem~\ref{theorem:p2} from
Theorem \ref{theorem:m-spaces}.

\noindent {\bf Proof of Theorem~\ref{theorem:p2}.}
Repeatedly apply Theorem \ref{theorem:m-spaces} and find an edge $e$ in $\H$
such that among those edges intersecting it there are at most $156$
pairwise disjoint ones. Then delete $e$ and those edges intersecting it from 
$\H$ and continue. If we can continue $k$ steps, then we find $k$ pairwise
disjoint edges. Otherwise, we decompose $\H$ into less than $k$ families,
$\H_{1}, \ldots, \H_{\ell}$,
of edges such that in each family $\H_{i}$ there are at most $156$ pairwise 
disjoint edges. 

The boundary of the union of $m$ half-spaces in $\mathbb{R}^3$ 
is the boundary of a polyhedron 
with at most $m$ facets, which in turn has complexity linear in $m$.
It now follows from Theorem \ref{theorem:p}
that each of the families $\H$ 
has fractional Helly number $2$ in a way that is independent of $P$,
as described in the statement of Theorem \ref{theorem:p}. 
It is well known that families of 
half-spaces (in any fixed dimension) have bounded VC-dimension.
Hence each $\H_{i}$ has a bounded VC-dimension (in fact bounded by $4$).
Therefore, each $\H_{i}$ has an $\epsilon$-net of size that depends only 
on $\epsilon$ (see \cite{HW87}).
The method of Alon and 
Kleitman \cite{AK92} implies that each $\H$ satisfies a $(p,2)$ theorem.
That is, if a subset $S$ of edges in $\H$ satisfies the $(p,2)$ property 
(that is, from every $p$ sets in $S$ there are $2$ sets that intersect), 
then there are $c(p)$ (a constant that depends only on $p$) 
vertices that together pierce all the sets in 
$S$. 

By our assumption, each $\H_{i}$ has the $(p,2)$-property for $p=157$. 
It follows that one can find a set of points of cardinality
at most $c(157)k$ that together pierce all the edges in $\H$.
\bbox

\section{The case of half-spaces in $\mathbb{R}^d$ where $d\geq 4$}
\label{sec:proofs3}
In this section we prove Theorem~\ref{theorem:gap}.

For every $n\in\mathbb{N}$ we need to construct a set $P$ of 
$N={n\choose 2}$ points
and a set $\F$ of $n$ half-spaces in $\mathbb{R}^4$ such that:
\begin{enumerate}
\item Every two edges in $\H(P,\F)$ have a non-empty intersection
\item Any subset of $P$ which pierce all edges in $\H(P,\F)$ must consist 
of at least $\frac{n-1}{2}$ points.
\end{enumerate}

The next lemma will be our main tool in constructing $\H(P,\F)$.
This lemma is a slight variation of an argument which was 
used by~\cite{AFR85} to upper bound the sign-rank of a hyper-graph.
\begin{lemma}\label{lemma:2d}
Let $\H$ be a hypergraph such that every $v \in V(\H)$
belongs to at most $d$ hyper-edges. Then $\H$ can be realized by points and
half-spaces in $\mathbb{R}^{2d}$. That is, $\H$ is isomorphic to
$\H(P,\F)$ for some set $P$ of points in $\mathbb{R}^{2d}$
and a family $\F$ of half spaces in $\mathbb{R}^{2d}$.
\end{lemma}

\begin{proof}
Pick some enumeration of $E(\H)$, $e_1,e_2,\dots,e_m$ where $m=|E(\H)|$.
For every $v\in V$ pick some real univariate polynomial $P_v(x)$
such that
\begin{itemize}
\item $P_v(0)=-1$,
\item $P_v(i) > 0$ if $v\in e_i$ and $P_v(i)<0$ if $v\notin e_i$, and
\item $\deg(P_v)\leq 2d$.
\end{itemize}
It is not hard to see that such a polynomial always exists: For example,
the polynomial 
$$
P_v(x)=-\frac{Q_v(x)}{Q_v(0)}, \mbox{ where } Q_v(x)=\prod_{i:v\in e_i}\left(x-(i+\frac{1}{4})\right)\left(x-(i-\frac{1}{4})\right)
$$ satisfies the above requirements.
For every $v\in V$ let $p_{v,i},i=0,\dots 2d$ denote the 
coefficients of $P_v(x)$. Notice that $p_{v,0}=-1$ for all $v$.

Every $v\in V$ will correspond to the point $x_v=(p_{v,1},\dots,p_{v,2d})$
and every $e_i$ correspond to the half-space 
$H_i=\{x: \langle x,n_i\rangle\geq 1\}$, where
$n_i=(i,i^2,\dots,i^{2d})$.
Observe that $\langle x_v,n_i\rangle = P_v(i)+1$ and therefore
$v \in e_i$ if and only if $x_v\in H_i$ as required.
\end{proof}

We now construct an hyper-graph $\H$ with ${n \choose 2}$ vertices
such that every vertex belongs to precisely two edges, every two edges 
have a non-empty intersection (that is, any matching in $\H$ is 
of size at most $1$), and finally, any set of vertices that pierces all
edges must consist of at least $\frac{n-1}{2}$ vertices.
Once we introduce such a hyper-graph, it follows from Lemma~\ref{lemma:2d}
that it can be realized in $\mathbb{R}^4$ by points and half-spaces.

We take the vertices of $\H$ to be the edges of a complete simple graph on 
$n$ vertices $K_{n}$. Let us denote the vertices of $K_{n}$ by 
$x_{1}, \ldots ,x_{n}$. Then $\H$ has ${n \choose 2}$ vertices. The hyper-graph
$\H$ will consist of $n$ edges $e_{1}, \ldots, e_{n}$ defined as follows.
For every $1 \leq i \leq n$ the edge $e_{i}$ is the collection of all edges in $K_{n}$ incident to $x_{i}$.

It is easy to check that indeed every two sets in $\H(P,\F)$
have a non-empty intersection and that any set of vertices of $\H$ 
that pierces
all the edges of $\H$ must have size of at least $\frac{n-1}{2}$, as desired.
\bbox



\end{document}